\documentclass[12pt]{amsart}
\usepackage{fullpage,hyperref}

\title{Markoff--Lagrange spectrum of one-sided shifts}

\author{Hajime Kaneko}
\address{Institute of Mathematics / Research Core for Mathematical Sciences, University of Tsukuba, Tsukuba, Japan}
\email{kanekoha@math.tsukuba.ac.jp}
\author{Wolfgang Steiner}
\address{Université Paris Cit\'e, CNRS, IRIF, F-75006 Paris, France}
\email{steiner@irif.fr}

\date{\today}
\thanks{This work was supported by the Agence Nationale de la Recherche through the projects CODYS (ANR-18-CE40-0007) and IZES (ANR-22-CE40-0011) and by the JSPS KAKENHI Grant Number 19K03439.}

\newtheorem{lemma}{Lemma}
\newtheorem{theorem}[lemma]{Theorem}
\newtheorem{proposition}[lemma]{Proposition}

\newcommand{\ba}{\mathbf{a}}
\newcommand{\bb}{\mathbf{b}}
\newcommand{\bc}{\mathbf{c}}
\newcommand{\be}{\mathbf{e}}
\newcommand{\bm}{\mathbf{m}}
\newcommand{\bsigma}{\boldsymbol{\sigma}}

\begin{document} 
\begin{abstract}
For the Lagrange spectrum and other applications, we determine the smallest accumulation point of binary sequences that are maximal in their shift orbits. 
This problem is trivial for the lexicographic order, and its solution is the fixed point of a substitution for the alternating lexicographic order. 
For orders defined by cylinders, we show that the solutions are $S$-adic sequences, where $S$ is a certain infinite set of substitutions that includes Sturmian morphisms. 
We also consider a similar problem for symmetric ternary shifts, which is applicable to the multiplicative version of the Markoff--Lagrange spectrum. 
\end{abstract}
\maketitle

\section{Introduction}
Let $A^\infty$ be the set of infinite words on a finite alphabet~$A$, equipped with a total order~$\le$ and the ultrametric~$d$ given by $d(a_1a_2\cdots,b_1b_2\cdots) = 2^{-\min\{n\ge 1\,:\,a_n\ne b_n\}}$ for $a_1a_2\cdots \ne b_1b_2\cdots$. 
We study properties of the set of \emph{$sup$-words}
\[
\mathcal{M}_\le := \{s_\le(\ba) \,:\, \ba \in A^\infty\}, \quad \mbox{with} \quad s_\le(a_1a_2\cdots) := \sup\nolimits_{n\ge1} a_na_{n+1}\cdots,
\]
for a large class of orders on~$A^\infty$.  
In particular, we are interested in the smallest accumulation point $\bm_\le$ of~$\mathcal{M}_\le$.
For the lexicographic order $\le_{\mathrm{lex}}$, words in $\mathcal{M}_{\le_{\mathrm{lex}}}$ occur as (quasi-greedy) $\beta$-expansions of~1 for real bases $\beta > 1$ (see \cite{Parry60}), with $\bm_{\le_{\mathrm{lex}}} = 1000\cdots$ being the limit as $\beta \to 1$.  
For the alternating lexicographic order~$\le_{\mathrm{alt}}$, most elements of $\mathcal{M}_{\le_{\mathrm{alt}}}$ are $(-\beta)$-expansions of $\frac{-\beta}{\beta+1}$ in the sense of \cite{Ito-Sadahiro09,Steiner13}, with $\bm_{\le_{\mathrm{alt}}}$ being the limit as $\beta \to 1$. 
An image of~$\mathcal{M}_{\le_{\mathrm{alt}}}$ occurs in a multiplicative version of the (Markoff--)Lagrange spectrum w.r.t.\ an integer base, which is defined in terms of well approximable numbers \cite{Dubickas06,Akiyama-Kaneko21}; see Proposition~\ref{p:Lbeta}.
Below the image of~$\bm_{\le_{\mathrm{alt}}}$, which is the fixed point of a substitution \cite{Allouche83,Allouche-Cosnard83,Dubickas07}, we find the discrete part of this spectrum.
The classical Markoff and Lagrange spectra are given by two-sided versions of $\mathcal{M}_{\le_{\mathrm{alt}}}$ (and the Lagrange spectrum is defined by $\limsup$ instead of $\sup$).
The unimodal order $\le_{\mathrm{uni}}$ yields kneading sequences of unimodal maps \cite{Milnor-Thurston88}, and $\bm_{\le_{\mathrm{uni}}}$ is the fixed point of the period-doubling (or Feigenbaum) substitution.
Sup-words are also closely related to infinite Lyndon words, which are
defined by $a_1 a_2 \cdots < a_n a_{n+1} \cdots$ for all $n \ge 2$;  see e.g.\ \cite{Postic-Zamboni20}.

We consider orders satisfying that
\begin{equation} \label{e:cylinderorder}
\ba \le \bb \le \bc \quad \mbox{implies} \quad d(\ba,\bb) \le d(\ba,\bc) \quad \mbox{for all $\ba, \bb, \bc \in A^\infty$}
\end{equation}
(note that $d(\ba,\bb) \le d(\ba,\bc)$ is equivalent to $d(\bb,\bc) \le d(\ba,\bc)$ by the strong triangle inequality), and we call them \emph{cylinder orders} because the elements of each cylinder of words are contiguous.
Here, the \emph{cylinder} (of length~$n$) given by $a_1\cdots a_n \in A^n$ is 
\[
[a_1\cdots a_n] := \{a'_1a'_2\cdots \in A^\infty \,:\, a'_1\cdots a'_n= a_1\cdots a_n\},
\]
and we write $[a_1\cdots a_n] < [b_1\cdots b_n]$ if $(a_1\cdots a_n)^\infty < (b_1\cdots b_n)^\infty$ (or, equivalently, $\ba < \bb$ for all $\ba \in [a_1\cdots a_n]$, $\bb \in [b_1\cdots b_n]$). 
Note that \eqref{e:cylinderorder} is equivalent to the condition (*) of \cite{Postic-Zamboni20}, and it includes generalized lexicographic orders as considered in \cite{Reutenauer06}.

In Section~\ref{sec:mark-lagr-spectr}, we give basic properties of $\mathcal{M}_{\le}$.
Section~\ref{sec:small-accum-point} contains our main result, an algorithm for determing the smallest accumulation point $\bm_\le$ of~$\mathcal{M}_\le$, for any cylinder order $(A^\infty, \le)$.
We also give a complete description, in terms of $S$-adic sequences, of all words $\bm_\le$ obtained by cylinder orders and of the discrete part of~$\mathcal{M}_{\le}$, i.e., all its elements below~$\bm_\le$.
Since the words $\bm_\le$ have linear factor complexity, real numbers having such $\beta$-expansions are in $\mathbb{Q}(\beta)$ or transcendental, for all Pisot or Salem bases $\beta \ge 2$.
In Section~\ref{sec:examples}, we determine $\bm_\le$ for some classical examples of cylinder orders and show that all maximal Sturmian sequences can occur. 
We consider cylinder orders on symmetric alphabets in Section~\ref{sec:symm-shift-spac} and apply our results to the multiplicative Lagrange spectrum and other problems in Section~\ref{sec:exampl-orders-symm},

\section{Markoff--Lagrange spectrum} \label{sec:mark-lagr-spectr}
We first show that the set of periodic words in $\mathcal{M}_\le$ is dense, and that $\mathcal{M}_\le$ is equal to 
\[
\mathcal{L}_\le := \{\ell_\le(\ba) \,:\, \ba \in A^\infty\}, \quad \mbox{with} \quad \ell_\le(a_1a_2\cdots) := \limsup\nolimits_{n\to\infty} a_na_{n+1}\cdots.
\]
Note that $\mathcal{M}_\le$ and $\mathcal{L}_\le$ can be seen as generalizations of the \emph{Markoff} and \emph{Lagrange spectrum} respectively. 
In the classical case, these spectra are defined by two-sided sequences, and 
the Lagrange spectrum is a strict subset of the Markoff spectrum \cite{Freiman68,Cusick-Flahive89}.

\begin{theorem} \label{t:ML}
Let $\le$ be a cylinder order on $A^\infty$.
Then
\[
\mathcal{L}_\le = \mathcal{M}_\le = \mathrm{cl}\{s_\le(\ba) \,:\, \ba \in A^\infty\, \mbox{purely periodic}\}.
\]
In particular, the set $\mathcal{M}_\le$ is closed. 
\end{theorem}

In the proof of the theorem, we use the following characterization of cylinder orders.

\begin{lemma} \label{l:cylinderorder2}
An order $\le$ on $A^\infty$ is a cylinder order if and only if
\begin{equation} \label{e:cylinderorder2}
\ba \le \bb \ \ \mbox{implies} \ \ \ba' \le \bb' \quad \mbox{for all $\ba,\bb,\ba', \bb' \in A^\infty$ with $\max(d(\ba,\ba'), d(\bb,\bb')) < d(\ba,\bb)$}.
\end{equation}
\end{lemma}

\begin{proof}
Let $\le$ be a cylinder order, and $\ba,\bb,\ba',\bb' \in A^\infty$ such that $\ba \le \bb$, $d(\ba,\ba') < d(\ba,\bb)$, $d(\bb,\bb') < d(\ba,\bb)$.
Then both $\ba \le \bb \le \ba'$ and $\bb' \le \ba' \le \bb$ are impossible by~\eqref{e:cylinderorder}, using that $d(\ba',\bb') = d(\ba,\bb)$ by the strong triangle inequality. 
This implies that $\bb > \ba'$ and thus $\bb' > \ba'$, i.e., \eqref{e:cylinderorder2} holds. 
Let now $\le$ be an order satisfying \eqref{e:cylinderorder2}.
Then $\ba \le \bb$ and $d(\ba,\bc) < d(\ba,\bb)$ imply that $\bc
< \bb$, thus $\ba \le \bb \le \bc$ with $d(\ba,\bc) < d(\ba,\bb)$ is impossible, i.e., \eqref{e:cylinderorder} holds. 
\end{proof}

\begin{proof}[Proof of Theorem~\ref{t:ML}]
We show first that $\mathcal{M}_\le$ is the closure of $s_\le(\ba')$ with purely periodic $\ba' \in A^\infty$.
Since $s_\le(s_\le(\ba)) = s_\le(\ba)$, it suffices to consider $\ba \in A^\infty$ with $s_\le(\ba) = \ba$.  
We have $(a_1 \cdots a_n)^\infty \in \mathcal{M}_\le$ whenever
\begin{equation} \label{e:rn}
2^{-n} d(\ba, a_{n+1}a_{n+2}\cdots) < 2^{-i} d(\ba, a_{i+1}a_{i+2}\cdots) \quad \mbox{for all}\ 1 \le i < n.
\end{equation}
Indeed, we have, for all $1 \le i < n$, that $a_{i+1} a_{i+2} \cdots \le \ba$ (because $s_\le(\ba) = \ba$) and 
\[
\begin{aligned}
d(\ba, a_{i+1} a_{i+2} \cdots) & > 2^{i-n} d(\ba, a_{n+1}a_{n+2}\cdots) = 2^{i-n} d((a_1\cdots a_n)^\infty, a_{n+1}a_{n+2}\cdots) \\
& = 2^i d((a_1\cdots a_n)^\infty, \ba) = d(a_{i+1} \cdots a_n (a_1\cdots a_n)^\infty, a_{i+1} a_{i+2} \cdots),
\end{aligned}
\]
hence \eqref{e:cylinderorder2} gives that $a_{i+1} \cdots a_n (a_1\cdots a_n)^\infty \le (a_1\cdots a_n)^\infty$, thus $(a_1 \cdots a_n)^\infty \in \mathcal{M}_\le$.
Since $\lim_{n\to\infty} 2^{-n} d(\ba, a_{n+1}a_{n+2}\cdots) = 0$, we have either $d(\ba, a_{n+1}a_{n+2}\cdots) = 0$ for some $n \ge 1$, i.e., $\ba$ is purely periodic, or infinitely many $n$ such that \eqref{e:rn} holds and thus $(a_1\cdots a_n)^\infty \in \mathcal{M}_\le$ infinitely often. 
Therefore, $\mathcal{M}_\le \subseteq \mathrm{cl}\{s_\le(\ba) : \ba \in A^\infty\, \mbox{purely periodic}\}$.
For the opposite inclusion, we have to show that $\mathcal{M}_\le$ is closed.
Consider $\ba = \lim_{k\to\infty} \ba^{(k)}$ with $\ba^{(k)} \in \mathcal{M}_\le$.
If $a_{n+1} a_{n+2} \cdots > \ba$ for some $n \ge 1$, then we also had
$a_{n+1}^{(k)} a_{n+2}^{(k)} \cdots > \ba^{(k)}$ for all large enough~$k$, contradicting that $\ba^{(k)} \in \mathcal{M}_\le$. 
This implies that $s_\le(\ba) = \ba$, i.e., $\ba \in \mathcal{M}_\le$. 

\smallskip
Since $s_\le(\ell_\le(\ba)) = \ell_\le(\ba)$ for all $\ba \in A^\infty$, we have $\mathcal{L}_\le \subseteq \mathcal{M}_\le$. 
For the opposite inclusion, let $\ba = \lim_{k\to\infty} \ba^{(k)}$ for some purely periodic words $\ba^{(k)} \in \mathcal{M}_\le$, let $(p_k)$ be an increasing sequence satisfying $\ba^{(k)} = (a_1^{(k)}\cdots a_{p_k}^{(k)})^\infty$, and let $\bb = a_1^{(1)}\cdots a_{p_1}^{(1)} a_1^{(2)}\cdots a_{p_2}^{(2)} \cdots$.
Then $\ell_\le(\bb) \ge \ba$.
If $a_{i+1}^{(k)}\cdots a_{p_k}^{(k)} a_1^{(k+1)}\cdots a_{p_{k+1}}^{(k+1)} \cdots > \ba$ for some $k \ge 1$, $0 \le i < p_k$, then 
\begin{equation} \label{e:dab}
\delta_{i,k} := d(a_{i+1}^{(k)}\cdots a_{p_k}^{(k)} a_1^{(k+1)}\cdots a_{p_{k+1}}^{(k+1)} \cdots, \ba) \le \max(d(\ba^{(k)}, \ba), d(\ba^{(k)}, \ba^{(k+1)}), 2^{-p_k}).
\end{equation}
Indeed, for $1 \,{\le}\, j \,{\le}\, p_k{-}i$, we cannot have $[a_{i+1}^{(k)} \cdots a_{i+j}^{(k)}] \,{>}\, [a_1 \cdots a_j] \,{=}\, [a_1^{(k)} \cdots a_j^{(k)}]$ because this contradicts $\ba^{(k)} \,{\in}\, \mathcal{M}_\le$, thus $\delta_{i,k} \,{=}\, 2^{-j}$ implies $d(\ba^{(k)}, \ba) \,{\ge}\, 2^{-j}$; similarly, for $p_k{-}i \,{<}\, j \,{\le}\, 2p_k{-}i$, 
\[
[a_{i+1}^{(k)} \cdots a_{p_k}^{(k)}a_1^{(k)} \cdots a_{i+j-p_k}^{(k)}] = [a_{i+1}^{(k)} \cdots a_{p_k}^{(k)}a_1^{(k+1)} \cdots a_{i+j-p_k}^{(k+1)}] > [a_1 \cdots a_j] = [a_1^{(k)} \cdots a_j^{(k)}]
\]
is impossible, thus $\delta_{i,k} = 2^{-j}$ implies $d(\ba^{(k)}, \ba^{(k+1)}) \ge 2^{p_k-i-j}$ or $d(\ba^{(k)}, \ba) \ge 2^{-j}$; hence, we have $\delta_{i,k} < 2^{i-2p_k}$ or $\delta_{i,k} \le 2^{i-p_k} d(\ba^{(k)}, \ba^{(k+1)})$ or $\delta_{i,k} \le d(\ba^{(k)}, \ba)$, which implies~\eqref{e:dab}. 
Since the right hand side of \eqref{e:dab} tends to 0 as $k \to \infty$, we have $\ell_\le(\bb) \le \ba$, thus $\ba = \ell_\le(\bb) \in \mathcal{L}_\le$. 
\end{proof}

\section{Smallest accumulation point of $\mathcal{M}_\le$} \label{sec:small-accum-point}
For determining the smallest accumulation point~$\bm_{\le}$ of $\mathcal{M}_\le$ for a cylinder order $(A^\infty,\le)$, we can restrict to two-letter alphabets, w.l.o.g., $A = \{0,1\}$.
Indeed, if $\{0,1\} \subseteq A$, $[0] < [1]$, then $\mathcal{M}_\le$ has the accumulation point $10^\infty$ (because $(10^n)^\infty \in \mathcal{M}_\le$ for all $n \ge 1$), and we clearly have $s_\le(\ba) > 10^\infty$ if $\ba \in A^\infty$ contains a letter $a_n \in A$ with $[a_n] > [1]$. 

We use substitutions (also called word morphisms) and limit words (or $S$-adic sequences). 
Let $A^*$ be the monoid of finite words over the alphabet~$A$, with concatenation as operation.
A~\emph{substitution} $\sigma: A^* \to A^*$ satisfies $\sigma(vw) = \sigma(v) \sigma(w)$ for all $v,w \in A^*$ and is extended naturally to~$A^\infty$; it suffices to give $\sigma(a)$ for $a \in A$ to define~$\sigma$. 
For a sequence $\bsigma = (\sigma_n)_{n\ge1}$ of substitutions on the alphabet~$A$ and an infinite word $\ba \in A^\infty$, the \emph{limit word} is
\[
\bsigma(\ba) := \lim_{n\to\infty} \sigma_{[1,n]}(\ba),
\]
if this limit exists; we use the notation $\sigma_{[1,n]} := \sigma_1 \circ \sigma_2 \circ \cdots \circ \sigma_n$ for $n \ge 0$, with $\sigma_{[1,0]}$ being the identity map.
For a set of substitutions~$S$, we denote the monoid generated by the composition of substitutions in~$S$ by~$S^*$.
We use the set of substitutions 
\[
S = \{\tau_{j,k} \,:\, 0 \le j < k\}, \quad \mbox{with} \quad \begin{array}{rl}\tau_{j,k}: & 0 \mapsto 10^j, \\[.5ex] & 1 \mapsto 10^k.\end{array}
\]

Our main result is the following characterization of the smallest accumulation point~$\bm_\le$ and the discrete part of $\mathcal{M}_{\le}$ for cylinder orders~$\le$. 

\begin{theorem} \label{t:main}
Let $\bm \in \{0,1\}^\infty$. 
Then $\bm = \bm_\le$ for some cylinder order $\le$ on $\{0,1\}^\infty$ with $0^\infty < 1^\infty$ if and only if $\bm = \sigma(10^\infty)$ for some $\sigma \in S^*$ or $\bm = \bsigma(1^\infty)$ for some $\bsigma \in S^\infty$. 

\smallskip
If $\bm_\le = \bsigma(1^\infty)$, $\bsigma = (\sigma_n)_{n\ge1} \in S^\infty$, then 
\begin{equation} \label{e:discrete}
\{\ba \in \mathcal{M}_\le \,:\, \ba < \bm_\le\} = \{(\sigma_{[1,n]}(0))^\infty \,:\, n \ge 0\}.
\end{equation}

If $\bm_\le = \sigma_{[1,h]}(10^\infty)$, $\sigma_1,\dots,\sigma_h \in S$, $h \ge 0$, then 
\begin{equation} \label{e:discrete2}
\begin{aligned}
\{\ba \in \mathcal{M}_\le \,:\, \ba < \bm_\le\} & = \{(\sigma_{[1,n]}(0))^\infty \,:\, 0 \le n \le h\} \\
& \quad \cup \{\sigma_{[1,h]}((10^j)^\infty) \,:\, j \ge 0,\, \sigma_{[1,h]}((10^j)^\infty) < \bm_\le\},
\end{aligned}
\end{equation}
and there is at most one $j \ge 0$ such that $\sigma_{[1,h]}((10^j)^\infty) < \bm_\le$.
\end{theorem}

The following proposition constitutes the core of the proof of Theorem~\ref{t:main} and provides an algorithm for calculating~$\bm_\le$. 

\begin{proposition} \label{p:main}
Let $\le$ be a cylinder order on $\{0,1\}^\infty$ with $[0] < [1]$. 

If $[10^j1] < [10^j0]$ holds for at most one $j \ge 0$, then $\bm_\le = 10^\infty$, and $\ba < \bm_\le$ implies that $\ba = 0^\infty$ or $\ba = (10^j)^\infty$ (in case $[10^j1] < [10^j0]$).

Otherwise, we have $\bm_\le = \tau_{j,k}(\bm_\preceq)$ for the cylinder order~$\preceq$ on $\{0,1\}^\infty$ defined by $\ba \preceq \bb$ if $\tau_{j,k}(\ba) \le \tau_{j,k}(\bb)$, where $j,k$ are minimal such that $0 \le j < k$, $[10^j1] < [10^j0]$, $[10^k1] < [10^k0]$; we have $[0] \prec [1]$ and $\{\ba \in \mathcal{M}_\le \,:\, \ba < \bm_\le\} = \{0^\infty\} \cup \{\tau_{j,k}(\ba) \,:\, \ba \in \mathcal{M}_\preceq,\, \ba \prec \bm_\preceq\}$. 
\end{proposition}

\begin{proof}
We have $\bm_\le \in [1]$ because $\mathcal{M}_\le \cap [0] = \{0^\infty\}$, and $\bm_\le \le 10^\infty$ because $(10^n)^\infty \in \mathcal{M}_\le$ for all $n \ge 0$. 
If $[10^j0] < [10^j1]$ for all $j \ge 0$, i.e., $\min\, [1] = 10^\infty$, then $\bm_\le = 10^\infty$.
Otherwise, let $j$ be minimal such that $[10^j1] < [10^j0]$, i.e., $\min [1] \in [10^j]$.
Then $\mathcal{M}_\le \cap [10^j1] = \{(10^j)^\infty\}$ implies that $\bm_\le \in [10^j0]$.
If $[10^k0] < [10^k1]$ for all $k > j$, i.e., $\min\, [10^j0] = 10^\infty$, then $\bm_\le = 10^\infty$.
Otherwise, let $k > j$ be minimal such that $[10^k1] < [10^k0]$, i.e., $\min [10^j0] \in [10^k1]$.
We have $\bm_\le \le \tau_{j,k}(10^\infty)$ because $\tau_{j,k}((10^n)^\infty) \in \mathcal{M}_\le$ for all $n \ge 0$, thus $\bm_\le \in [10^k1]$.
Each word in $\mathcal{M}_\le \cap [10^k1]$ is a concatenation of blocks $10^j$ and~$10^k$, thus $\mathcal{M}_\le \cap [10^k1] \subseteq \tau_{j,k}([1])$.
This proves that $\bm_\le = \tau_{j,k}(\bm_\preceq)$, where $\ba \preceq \bb$ if $\tau_{j,k}(\ba) \le \tau_{j,k}(\bb)$; note that $\bm_\preceq \in [1]$ because $[0] \prec [1]$.
Since $\tau_{j,k}([w0]) \subseteq [\tau_{j,k}(w0)1]$ and $\tau_{j,k}([w1]) \subseteq [\tau_{j,k}(w0)0]$ for all $w \in \{0,1\}^*$, $\preceq$~is a cylinder order.
If $\ba \in \mathcal{M}_\le$ with $\ba < \bm_\le$, then $\ba = 0^\infty$ or $\ba = (10^j)^\infty = \tau_{j,k}(0^\infty)$ or $\ba = \tau_{j,k}(\ba')$ with $\ba' \in \mathcal{M}_\preceq \cap [1]$, $\ba' \prec \bm_\preceq$. 
\end{proof}

The following lemma is used in the construction of a cylinder order $\le$ such that $\bm_\le = \bm$ for a given word~$\bm$.  
Here, $\varepsilon$ denotes the empty word.

\begin{lemma} \label{l:wn}
Let $\bsigma = (\sigma_n)_{n\ge1} \in S^\infty$, $w_n = \sigma_{[1,n]}(0) \cdots \sigma_{[1,1]}(0)$ for $n \ge 0$, with $w_0 = \varepsilon$. 

For all even $n \ge 0$, we have $\sigma_{[1,n]}([0]) \subseteq [w_n0]$, $\sigma_{[1,n]}([1]) \subseteq [w_n1]$. 

For all odd $n \ge 1$, we have $\sigma_{[1,n]}([0]) \subseteq [w_n1]$, $\sigma_{[1,n]}([1]) \subseteq [w_n0]$. 
\end{lemma}

\begin{proof}
Since $\sigma_{[1,0]}$ is the identity, the statement is trivial for $n = 0$.
Suppose that it is true for all $\bsigma \in S^\infty$ for some even $n \ge 0$. 
Then
\[
\begin{aligned}
\sigma_{[1,n+1]}([0]) & = \sigma_1 \circ \sigma_{[2,n+1]}([0]) \subseteq \sigma_1([\sigma_{[2,n+1]}(0) \cdots \sigma_{[2,2]}(0) 0]) \subseteq [\sigma_{[1,n+1]}(0) \cdots \sigma_{[1,2]}(0) \sigma_1(0) 1], \\
\sigma_{[1,n+1]}([1]) & = \sigma_1 \circ \sigma_{[2,n+1]}([1]) \subseteq \sigma_1([\sigma_{[2,n+1]}(0) \cdots \sigma_{[2,2]}(0) 1]) \subseteq [\sigma_{[1,n+1]}(0) \cdots \sigma_{[1,2]}(0) \sigma_1(0) 0],
\end{aligned}
\]
thus the statement is true for all $\bsigma \in S^\infty$ for $n{+}1$.
The case of odd~$n$ is similar. 
\end{proof}

\begin{proof}[Proof of Theorem~\ref{t:main}]
Let first $\bm_{\le}$ be a cylinder order with $[0] < [1]$. 
By iterating Proposition~\ref{p:main}, we obtain a finite sequence $\sigma_1, \dots, \sigma_h \in S$ such that $\bm_\le = \sigma_{[1,h]}(10^\infty)$ or an infinite sequence $\bsigma = (\sigma_n)_{n\ge1} \in S^\infty$ such that $\bm_\le \in \sigma_{[1,n]}([1])$ for all $n \ge 1$.
Since $\sigma_{[1,n+1]}(1)$ starts with $\sigma_{[1,n]}(1)$ and is longer than $\sigma_{[1,n]}(1)$ for all $n \ge 1$, we have $\bigcap_{n\ge1} \sigma_{[1,n]}([1]) = \{\bsigma(1^\infty)\}$.
Equations \eqref{e:discrete} and \eqref{e:discrete2} respectively follow from Proposition~\ref{p:main}.

Let now $\bsigma = (\sigma_n)_{n\ge1} = (\tau_{j_n,k_n})_{n\ge1} \in S^\infty$.
By Proposition~\ref{p:main} and Lemma~\ref{l:wn}, we have $\bsigma(1^\infty) = \bm_\le$ for all cylinder orders $\le$ satisfying
\begin{equation} \label{e:wnorder}
\begin{array}{cl}
[\sigma_{[1,n]}(10^i)w_n0] < [\sigma_{[1,n]}(10^i)w_n1] & \text{for all even $n\ge0$, $j_{n+1} \ne i < k_{n+1}$}, \\
& \text{and for all odd $n\ge1$, $i \in \{j_{n+1}, k_{n+1}\}$}, \\[.5ex]
{}[\sigma_{[1,n]}(10^i)w_n1] < [\sigma_{[1,n]}(10^i)w_n0] & \text{for all odd $n\ge1$, $j_{n+1} \ne i < k_{n+1}$}, \\
& \text{and for all even $n\ge0$, $i \in \{j_{n+1}, k_{n+1}\}$}.
\end{array}
\end{equation}
Such cylinder orders exist since $\sigma_{[1,n+1]}(1)w_{n+1}$ is longer than $\sigma_{[1,n]}(10^{k_{n+1}})w_n = \sigma_{[1,n+1]}(1)w_n$ for all $n \ge 0$.
To obtain $\bm_\le = \sigma_{[1,h]}(10^\infty)$, $h \ge 0$, we use cylinder orders~$\le$ such that \eqref{e:wnorder} holds only for $n < h$ and such that, for all $i \ge 0$, $[\sigma_{[1,h]}(10^i)w_h0] < [\sigma_{[1,h]}(10^i)w_h1]$ if $h$ is even, $[\sigma_{[1,h]}(10^i)w_h1] < [\sigma_{[1,h]}(10^i)w_h0]$ if $h$ is odd. 
\end{proof}

The sequence $\bm_{\le}$ is thus either eventually periodic or an $S$-adic sequence.
A~word $\ba = a_1a_2\cdots$ is eventually periodic if and only if the factor complexity
\[
p_{a_1a_2\cdots}(n) := \#\{a_{k+1} a_{k+2} \cdots a_{k+n} \,:\, k \ge 0\}
\]
is bounded; see e.g.\ \cite[Theorem~1.3.13]{Lothaire02}.
The smallest complexity for an aperiodic sequence is $p_{\ba}(n)  = n{+}1$, which is attained precisely by Sturmian sequences; see e.g.\ \cite[Theorem~2.1.5]{Lothaire02}.
By \cite[Proposition~2.1]{Creutz-Pavlov23}, all aperiodic words with $\limsup \frac{p_{\ba}(n)}{n} < \frac{4}{3}$ are essentially equal to~$\bsigma(1^\infty)$ with $\bsigma = (\tau_{j_n,k_n})_{n\ge1} \in S^\infty$ (and $k_n \le 2j_n{+}1$ or $(j_n,k_n) = (0,2)$).
Without conditions on $j_n,k_n$, we get the following upper bound for $p_{\bsigma(1^\infty)}(n)$, which is optimal since $p_{\tau_{0,k}^\infty(1^\infty)}(n) = 3n{-}2$ for all $k \ge 2$, $2 \le n \le k$. 

\begin{proposition} \label{p:linear}
Let $\le$ be a cylinder order. 
Then $p_{\bm_\le}(n) \le 3n{-}2$ for all $n \ge 2$.
\end{proposition}

\begin{proof}
The proof is similar to that of \cite[Theorem~17]{Balkova06}; see also \cite[Proposition~4.1]{Creutz-Pavlov23}. 
Recall that the set of factors of a word $\ba = a_1a_2\cdots \in \{0,1\}^\infty$ is $\{a_{k+1} a_{k+2} \cdots a_{k+n} : k,n \ge 0\}$, a factor $v$ of~$\ba$ is \emph{strong bispecial} if all four words $0v0, 0v1, 1v0, 1v1$ are factors of~$\ba$, \emph{weak bispecial} if $0v1, 1v0$ are factors and $0v0, 1v1$ are not factors of~$\ba$.
Then $v$ is a strong/weak bispecial factor of $\bsigma(1^\infty)$, $\bsigma = (\sigma_n)_{n\ge1} = (\tau_{j_n,k_n})_{n\ge1} \in S^\infty$, if and only if $v = 0^{j_1} /\, 0^{k_1-1}$, $k_1 \ge j_1{+}2$, or $v = 0^{j_1}\sigma_1(v'0)$, where $v'$ is a strong/weak bispecial factor of $\lim_{n\to\infty}\bsigma_{[2,n]}(1^\infty)$.
By iterating, we obtain that all strong/weak bispecial factors of $\bsigma(1^\infty)$ are of the form
\[
w_{\ell,h} = \sigma_{[1,0]}(0^{j_1}) \cdots \sigma_{[1,h-1]}(0^{j_h}) \sigma_{[1,h]}(0^\ell) \sigma_{[1,h]}(0) \cdots \sigma_{[1,1]}(0),  
\]
with $h \ge 0$ such that $k_{h+1} \ge j_{h+1}{+}2$, where $\ell = j_{h+1}$ for a strong bispecial factor, $\ell = k_{h+1}{-}1$ for a weak one; here, $w_{\ell,0} = 0^\ell$. 
For any recurrent word $\ba \in \{0,1\}^\infty$, the difference of  $p_{\ba}(n{+}2)-p_{\ba}(n{+}1)$ and $p_{\ba}(n{+}1)-p_{\ba}(n)$ equals the difference of the number of strong and weak bispecial factors of~$\ba$ of length~$n$; see \cite[Proposition~3.2]{Cassaigne97}. 
By telescoping and since $p_{\ba}(1)-p_{\ba}(0)=1$, $p_{\ba}(n{+}2)-p_{\ba}(n{+}1)-1$ is equal to the difference of the number of strong and weak bispecial factors of~$\ba$ up to length~$n$.
Since $|w_{k_{h+1}-1,h}| < |\sigma_{[1,h+1]}(1)| \le |\sigma_{[1,h+2]}(0)| < |w_{j_{h+3},h+2}|$ for all $h \ge 0$, this difference for $\ba = \bsigma(1^\infty)$ is at most~$2$ for all $n \ge 0$; note that $\bsigma(1^\infty)$ is recurrent since $\sigma_{[n,n+1]}(1)$ starts with  $10^{k_n}1$ for all $n \ge 1$. 
Since $p_{\ba}(2) \le 4$, we have thus $p_{\bsigma(1^\infty)} \le 3n{-}2$ for all $n \ge 2$. 
Since $p_{\sigma_{[1,h]}(10^\infty)}(n) \le p_{\bsigma(1^\infty)}(n)$ if $k_{h+1}$ is sufficiently large, we also have $p_{\sigma_{[1,h]}(10^\infty)}(n) \le 3n{-}2$ for all $n \ge 2$, which proves the proposition by Theorem~\ref{t:main}.
\end{proof}

Since the factor complexity is bounded by a linear function, we can apply the results of \cite{Adamczewsk-Bugeaud07}.
For $\beta > 1$, let
\[
\pi_\beta(a_1a_2\cdots) := \sum_{n=1}^\infty \frac{a_n}{\beta^n}.
\]
Recall that Pisot and Salem numbers are algebraic integers $\beta > 1$ with all Galois conjugates (except $\beta$ itself) having absolute value $\le 1$; $\beta$ is a Salem number if a conjugate lies on the unit circle, a Pisot number otherwise. 
In particular, all integers $\beta \ge 2$ are Pisot numbers. 

\begin{proposition} \label{p:transcendental}
Let $\beta \ge 2$ be a Pisot or Salem number, and let $\le$ be a cylinder order on $\{0,1\}^\infty$.  
Then $\pi_\beta(\bm_\le)$ is in $\mathbb{Q}(\beta)$ or transcendental. 
\end{proposition}

\begin{proof}
This is a direct consequence of Proposition~\ref{p:linear} and \cite[Theorem~1A]{Adamczewsk-Bugeaud07}. 
\end{proof}

\section{Examples} \label{sec:examples}
\subsection{Lexicographic order}
The classical order on the set of infinite words $A^\infty$ with an ordered alphabet $(A,\le)$ is the lexicographic order, defined by $[wa] <_{\mathrm{lex}} [wb]$ for all $w \in A^*$, $a,b \in A$ with $a < b$.
For $A = \{0,1\}$, we have $\bm_{\le_{\mathrm{lex}}} = 10^\infty = \min \mathcal{M}_{\le_{\mathrm{lex}}} \setminus \{0^\infty\}$. 

By \cite[\S 2]{Parry60}, a sequence $\ba \in \{0,1\}^\infty$ is the greedy $\beta$-expansion of~$1$ for some $\beta \in (1,2)$ if and only if $\ba \in \mathcal{M}_{\le_{\mathrm{lex}}} \setminus \{10^\infty\}$ and $\ba$ is not purely periodic; it is the quasi-greedy $\beta$-expansion of~$1$ for some $\beta \in (1,2]$ if and only if $\ba \in \mathcal{M}_{\le_{\mathrm{lex}}}$ and $\ba$ does not end with~$0^\infty$.
Here, the greedy $\beta$-expansion of~$1$ is the lexicographically largest sequence $\ba \in \{0,1\}^\infty$ with $\pi_\beta(\ba) = 1$, the quasi-greedy $\beta$-expansion of~$1$ is the largest such sequence that does not end with~$0^\infty$. 
By \cite[Theorem~1]{Hubbard-Sparrow90}, we also have that $(0^\infty,\ba)$ is the pair of kneading sequences of a Lorenz map if and only if $\ba \in \mathcal{M}_{\le_{\mathrm{lex}}}$ does not end with $0^\infty$.

\subsection{Alternating lexicographic order} \label{sec:altern-lexic-order}
The alternating lexicographic order is defined by $[wa] <_{\mathrm{alt}} [wb]$ if $a < b$ and $|w|$ is even, or $a > b$ and $|w|$ is odd, where $|w|$ denotes the length of a word $w \in A^*$.
For $A = \{0,1\}$ with $0 < 1$, we have $[11] <_{\mathrm{alt}} [10]$, $[101] >_{\mathrm{alt}} [100]$, $[1001] <_{\mathrm{alt}} [1000]$, thus $\bm_{\le_{\mathrm{alt}}} = \tau_{0,2}(\bm_{\preceq})$ by Proposition~\ref{p:main}, with $\ba \preceq \bb$ if $\tau_{0,2}(\ba) \le_{\mathrm{alt}} \tau_{0,2}(\bb)$. 
Since $\preceq$ is equal to~$\le_{\mathrm{alt}}$, we obtain that $\bm_{\le_{\mathrm{alt}}}$ is the fixed point of $\tau_{0,2}$, i.e., 
\[
\bm_{\le_{\mathrm{alt}}} = \tau_{0,2}(\bm_{\le_{\mathrm{alt}}}) = 100111001001001110011\cdots;
\] 
see also \cite{Allouche83,Allouche-Cosnard83,Dubickas07}.
By Theorem~\ref{t:main}, we have 
\begin{equation} \label{e:discretealt}
\{\ba \in \mathcal{M}_{\le_{\mathrm{alt}}} \,:\, \ba <_{\mathrm{alt}} \bm_{\le_{\mathrm{alt}}}\} = \{(\tau_{0,2}^n(0))^\infty \,:\, n \ge 0\} = \{0^\infty, 1^\infty, (100)^\infty, (10011)^\infty, \dots\}.
\end{equation}

According to \cite{Ito-Sadahiro09}, the $(-\beta)$-expansion of $x \in \big[\frac{-\beta}{\beta+1},\frac{1}{\beta+1}\big)$, $\beta > 1$, is the sequence $a_1a_2\cdots$ with $a_n = \big\lfloor \frac{\beta}{\beta+1}-\beta T_{-\beta}^{n-1}(x)\big\rfloor$, given by the $(-\beta)$-transformation $T_{-\beta}(y) := -\beta y - \big\lfloor \frac{\beta}{\beta+1}-\beta y\big\rfloor$, and the set of $(-\beta)$-expansions is characterized by that of $\frac{-\beta}{\beta+1}$.
By \cite[Theorem~2]{Steiner13}, a sequence $\ba$ is the $(-\beta)$-expansion of $\frac{-\beta}{\beta+1}$ for some $\beta \in (1,2)$ if and only if $\ba \in \mathcal{M}_{\le_{\mathrm{alt}}} \setminus \{(10)^\infty\}$, $\ba >_{\mathrm{alt}} \bm_{\le_{\mathrm{alt}}}$, and $\ba \notin \{w1,w00\}^\infty \setminus \{(w1)^\infty\}$ for all $w \in \{0,1\}^*$ such that $(w1)^\infty >_{\mathrm{alt}} \bm_{\le_{\mathrm{alt}}}$.
Note that continued fractions are also ordered by the alternating lexicographic order on the sequences of partial quotients, and $\bm_{\le_{\mathrm{alt}}}$ occurs e.g.\ in \cite[Remark~11.1]{Kraaikamp-Schmidt-Steiner12}. 

\subsection{Unimodal maps}
Let $A = \{0,1\}$ and define the \emph{unimodal order} by $[w0] <_{\mathrm{uni}} [w1]$ if $|w|_1$ is even, $[w0] >_{\mathrm{uni}} [w1]$ if $|w|_1$ is odd, where $|w|_1$ denotes the number of occurrences of~1 in $w \in \{0,1\}^*$.
Then we have $[11] <_{\mathrm{uni}} [10]$, $[101] <_{\mathrm{uni}} [100]$, and
\[
\bm_{\le_{\mathrm{uni}}} = \tau_{0,1}(\bm_{\le_{\mathrm{alt}}}) = 10111010101110111011101010111010\cdots.
\] 
This is the fixed point of the period-doubling (or Feigenbaum) substitution $0 \mapsto 11$, $1 \mapsto 10$.
The set $\mathcal{M}_{\le_{\mathrm{uni}}}$ is the set of kneading sequences of unimodal maps \cite{Collet-Eckmann80,Milnor-Thurston88}. 

We define the \emph{flipped unimodal order} by $[w0] <_{\mathrm{flip}} [w1]$ if $|w|_0$ is even, $[w0] >_{\mathrm{flip}} [w1]$ if $|w|_0$ is odd, where $|w|_0$ denotes the number of occurrences of~0 in $w \in \{0,1\}^*$. 
Then we have $[11] >_{\mathrm{flip}} [10]$, $[101] <_{\mathrm{flip}} [100]$, $[1001] >_{\mathrm{flip}} [1000]$, $[10001] >_{\mathrm{flip}} [10000]$, and
\[
\bm_{\le_{\mathrm{flip}}} = \tau_{1,3}(\bm_{\le_{\mathrm{alt}}}) =  100010101000100010001010100010101000101010001000\cdots.
\] 
Note that $0 \bm_{\le_{\mathrm{flip}}} = F(\bm_{\le_{\mathrm{uni}}})$, where $F(a_1a_2\cdots) := (1{-}a_1)(1{-}a_2)\cdots$, and we have 
\[
\mathcal{M}_{\le_{\mathrm{uni}}} = 1F(\mathcal{M}_{\le_{\mathrm{flip}}}) \cup \{0^\infty\}.
\]

\subsection{Sturmian sequences} \label{sec:sturmian-sequences}
The set of substitutions $\{\theta_k : k \ge 1\}$ defined by $\theta_k(0) = 0^{k-1}1$, $\theta_k(1) = 0^{k-1}10$, generates the standard Sturmian words; see \cite[Corollary~2.2.22]{Lothaire02}.
Since $\tau_{k-1,k}$ is rotationally conjugate to $\theta_k$, more precisely $\theta_k(w)0^{k-1} = 0^{k-1}\tau_{k-1,k}(w)$ for all $w \in \{0,1\}^*$, the set of substitutions $\{\tau_{k-1,k} : k \ge 1\}$ generates the same shifts as $\{\theta_k : k \ge 0\}$.
Therefore, the limit words of sequences in $\{\tau_{k-1,k} : k \ge 1\}^\infty$ provide elements of all Sturmian shifts. 
For example, the limit word of the sequence $(\tau_{0,1})^\infty$ is the Fibonacci word.

\section{Symmetric alphabets} \label{sec:symm-shift-spac}
For a real number $q > 1$, the set 
\begin{equation} \label{e:Lbeta}
\tilde{\mathcal{L}}_q := \big\{\limsup_{n\to\infty} \|x q^n\| \,:\, x \in \mathbb{R}\big\},
\end{equation}
where $\|.\|$ denotes the distance to the nearest integer, is a multiplicative version of the Lagrange spectrum and was studied in \cite{Dubickas06,Akiyama-Kaneko21}.
If $q$ is an integer, then representing $x = \sum_{k=-\infty}^\infty a_k q^{-k}$ with $a_k \in \mathbb{Z}$, $a_k \ne 0$ for finitely many $k \le 0$, $|\sum_{k=n+1}^\infty a_k q^{n-k}| \le 1/2$, gives that $\|x q^n\| = |\sum_{k=n+1}^\infty a_k q^{n-k}|$; see also Proposition~\ref{p:Lbeta} below.
This leads us to consider
\[
\mathcal{M}^{\mathrm{abs}}_\le = \{s^{\mathrm{abs}}_\le(\ba) \,:\, \ba \in \{0,\pm1\}^\infty\} \quad \mbox{with} \quad s^{\mathrm{abs}}_\le(a_1a_2\cdots) = \sup\nolimits_{n\ge1} \mathrm{abs}(a_na_{n+1}\cdots), 
\]
where 
\[
\mathrm{abs}(\ba) = \begin{cases}\ba & \mbox{if}\ \ba \ge_{\mathrm{lex}} 0^\infty, \\ -\ba & \mbox{if}\ \ba \le_{\mathrm{lex}} 0^\infty,\end{cases}, \quad -(a_1a_2\cdots) = (-a_1)(-a_2)\cdots.
\]
We denote the smallest accumulation point of $\mathcal{M}^{\mathrm{abs}}_\le$ by~$\bm^{\mathrm{abs}}_\le$.

The same proof as for Theorem~\ref{t:ML} shows for all cylinder orders $\le$ on $\{0,\pm1\}^\infty$ that 
\[
\mathcal{L}^{\mathrm{abs}}_\le = \mathcal{M}^{\mathrm{abs}}_\le = \mathrm{cl}\{s^{\mathrm{abs}}_\le(\ba) \,:\, \ba \in \{0,\pm1\}^\infty\, \mbox{purely periodic}\},
\]
where $\mathcal{L}^{\mathrm{abs}}_\le := \{\limsup_{n\ge1} \mathrm{abs}(a_na_{n+1}\cdots) \,:\, a_1a_2\cdots \in \{0,\pm1\}^\infty\}$.

In the following, we assume that a cylinder order on $\{0,\pm1\}^\infty$ is \emph{consistent} (with the natural order on $\{0,\pm1\}$), which means that, for each $w \in \{0,\pm1\}^*$, we have $[w(-1)] < [w0] < [w1]$ or $[w(-1)] > [w0] > [w1]$.
In order to describe~$\bm^{\mathrm{abs}}_\le$, we define maps $\varrho_0, \varrho_1, \varrho_2$ from $\{0,1\}^*$ to $\{0,\pm1\}^*$ by $\varrho_0(\varepsilon) = \varrho_1(\varepsilon) = \varrho_2(\varepsilon) = \varepsilon$ for the empty word $\varepsilon$, and
\[
\begin{aligned}
\varrho_0(w0) & = \begin{cases}\varrho_0(w) 1 & \mbox{if $|w|_0$ is even}, \\ \varrho_0(w) (-1) & \mbox{if $|w|_0$ is odd},\end{cases} & \varrho_0(w1) & = \begin{cases}\varrho_0(w) 10 & \mbox{if $|w|_0$ is even}, \\ \varrho_0(w) (-1)0 & \mbox{if $|w|_0$ is odd},\end{cases} \\ 
\varrho_1(w0) & = \begin{cases}\varrho_1(w) 1 & \mbox{if $|w|_1$ is even}, \\ \varrho_1(w) (-1) & \mbox{if $|w|_1$ is odd},\end{cases} & \varrho_1(w1) & = \begin{cases}\varrho_1(w) 10 & \mbox{if $|w|_1$ is even}, \\ \varrho_1(w) (-1)0 & \mbox{if $|w|_1$ is odd},\end{cases} \\
\varrho_2(w0) & = \begin{cases}\varrho_2(w) 1 & \mbox{if $|w|$ is even}, \\ \varrho_2(w) (-1) & \mbox{if $|w|$ is odd},\end{cases} & \varrho_2(w1) & = \begin{cases}\varrho_2(w) 10 & \mbox{if $|w|$ is even}, \\ \varrho_2(w) (-1)0 & \mbox{if $|w|$ is odd},\end{cases} \label{e:rho2}
\end{aligned}
\]
for all $w \in \{0,1\}^*$, where $|w|$ is the length of a word $w$ and $|w|_i$ the number of occurrences of the letter~$i$ in~$w$. 
As for substitutions, the maps $\varrho_i$ are extended naturally to $\{0,1\}^\infty$. 

\begin{theorem} \label{t:sym}
Let $\bm \in \{0,\pm1\}^\infty$.
Then $\bm  = \bm^{\mathrm{abs}}_\le$ for some consistent cylinder order $\le$ on $\{0,\pm1\}^\infty$ with $0^\infty < 1^\infty$ if and only if $\bm = \sigma(\bm_\preceq)$ for some $\sigma \in \{\varrho_0, \varrho_1, \varrho_2, \tau_{0,1}\}$ and some cylinder order $\preceq$ on $\{0,1\}^\infty$ with $0^\infty \prec 1^\infty$.

If $\bm^{\mathrm{abs}}_\le = \sigma(\bm_\preceq)$, then we can assume that $\ba \preceq \bb$ if and only if $\sigma(\ba) \le \sigma(\bb)$, and we have
\begin{equation} \label{e:discabs}
\{\ba \in \mathcal{M}^{\mathrm{abs}}_\le : \ba < \bm^{\mathrm{abs}}_\le\} = \{0^\infty\} \cup \{\sigma(\ba) : \ba \in \mathcal{M}_\preceq,\, \ba \prec \bm_\preceq\}.
\end{equation}
\end{theorem}

\begin{proof}
Let first $\le $ be a consistent cylinder order on $\{0,\pm1\}^\infty$.
Then the order~$\preceq$ defined by $\ba \preceq \bb$ if $\sigma(\ba) \le \sigma(\bb)$ is a cylinder order for all $\sigma \in \{\varrho_0, \varrho_1, \varrho_2, \tau_{0,1}\}$. 
Indeed, for any $w\in\{0,1\}^*$, we have 
$\sigma([w0])\subset[\sigma(w)xy], \ 
\sigma([w1])\subset[\sigma(w)x0]$, where $x,y\in \{\pm1\}$.

Assume first that $[1\bar{1}] < [10]$; here and in the following, we use the notation $\bar{1} = -1$. 
Then $\mathcal{M}^{\mathrm{abs}}_\le \cap [1\bar{1}] = \{(1\bar{1})^\infty\}$, thus $\bm^{\mathrm{abs}}_\le \in [10]$.
If $[10\bar{1}] < [100]$, then each $1$ in a word in $\mathcal{M}^{\mathrm{abs}}_\le \cap [10\bar{1}]$ is followed by $\bar{1}$ or $0\bar{1}$, and each $\bar{1}$ is followed by $1$ or $01$, i.e., 
\[
\mathcal{M}^{\mathrm{abs}}_\le \cap [10\bar{1}] \subseteq 10\,(\{\bar{1},\bar{1}0\}\{1,10\})^\infty = \varrho_2([1]).
\]
Therefore, $\bm^{\mathrm{abs}}_\le = \varrho_2(\bm_\preceq)$ for the cylinder order~$\preceq$ defined by $\ba \preceq \bb$ if $\varrho_2(\ba) \le \varrho_2(\bb)$ .
If $[101] < [100]$, then each $1$ in a word in $\mathcal{M}^{\mathrm{abs}}_\le \cap [101]$ is followed by $\bar{1}$ or $01$, and each $\bar{1}$ is followed by $1$ or $0\bar{1}$, i.e., 
\[
\begin{aligned}
\mathcal{M}^{\mathrm{abs}}_\le \cap [101] & \subseteq 10((10)^*1(\bar{1}0)^*\bar{1})^\infty \cup 10((10)^*1(\bar{1}0)^*\bar{1})^*(10)^\infty \cup 10((10)^*1(\bar{1}0)^*\bar{1})^*(10)^*1(\bar{1}0)^\infty \\
& = \varrho_0(1(1^*01^*0)^\infty) \cup \varrho_0(1(1^*01^*0)^* 1^\infty) \cup \varrho_0(1(1^*01^*0)^* 1^*0 1^\infty) = \varrho_0([1]).
\end{aligned}
\]
Therefore, we have $\bm^{\mathrm{abs}}_\le = \varrho_0(\bm_\preceq)$  for the cylinder order~$\preceq$ defined by $\ba \preceq \bb$ if $\varrho_0(\ba) \le \varrho_0(\bb)$.
Assume now $[11] < [10]$.
Then $\mathcal{M}^{\mathrm{abs}}_\le \cap [11] = \{1^\infty\}$, thus $\bm_\le \in [10]$.
If $[10\bar{1}] < [100]$, then 
\[
\mathcal{M}^{\mathrm{abs}}_\le \cap [10\bar{1}] \subseteq 10(\bar{1}^*\bar{1}01^*10)^\infty 
\cup 10(\bar{1}^*\bar{1}01^*10)^* \bar{1}^\infty \cup 10(\bar{1}^*\bar{1}01^*10)^*  \bar{1}^*\bar{1}0 1^\infty = \varrho_1([1]),
\]
thus $\bm^{\mathrm{abs}}_\le = \varrho_1(\bm_\preceq)$, with $\preceq$ defined by $\ba \preceq \bb$ if $\varrho_1(\ba) \le \varrho_1(\bb)$.
If $[101] < [100]$, then 
\[
\mathcal{M}^{\mathrm{abs}}_\le \cap [101] \subseteq 10\{1,10\}^\infty = \tau_{0,1}([1]),
\]
thus $\bm^{\mathrm{abs}}_\le = \tau_{0,1}(\bm_\preceq)$ for the cylinder order~$\preceq$ defined by $\ba \preceq \bb$ if $\tau_{0,1}(\ba) \le \tau_{0,1}(\bb)$.
Since $\varrho_0(0^\infty) = \varrho_2(0^\infty) =  (1\overline{1})^\infty$ and $\varrho_1(0^\infty) = \tau_{0,1}(0^\infty) = 1^\infty$, equation \eqref{e:discabs} holds.

Let now $\preceq$ be a cylinder order on $\{0,1\}^\infty$ with $0^\infty \prec 1^\infty$ and $\sigma \in \{\varrho_0, \varrho_1, \varrho_2, \tau_{0,1}\}$.
Then there exists a consistent cylinder order $\le$ on $\{0,\pm1\}^\infty$ satisfying $\sigma(\ba) \le \sigma(\bb)$ if $\ba \preceq \bb$ and $0^\infty < 1^\infty$.
Indeed, for $w \in \{0,\pm1\}^*$ and distinct $a, b \in \{0,\pm1\}$, we set $[wa] < [wb]$ if $\ba' \prec \bb'$ for some $\ba', \bb' \in \{0,1\}^\infty$ with $\sigma(\ba') \in [wa]$, $\sigma(\bb') \in [wb]$; by Lemma~\ref{l:cylinderorder2}, this does not depend on the choice of $\ba', \bb'$. 
Moreover, since $\sigma(\ba) \in [w\bar{1}]$ and $\sigma(\bb) \in [w1]$ is impossible, we have no obstruction to a consistent cylinder order.
Since $0^\infty \prec 1^\infty$, we have $[1\overline{1}] < [10]$ in case $\sigma \in \{\varrho_0, \varrho_2\}$, $[11] < [10]$ in case $\sigma \in \{\varrho_1, \tau_{0,1}\}$.
We can set $[10\bar{1}] < [100]$ in case $\sigma \in \{\varrho_1, \varrho_2\}$ because $([100] \cup [101]) \cap \sigma(\{0,1\})^\infty = \emptyset$, similarly we can set $[101] < [100]$ in case $\sigma \in \{\varrho_0, \tau_{0,1}\}$.
Then we have $\bm^{\mathrm{abs}}_\le = \sigma(\bm_\le)$. 
\end{proof}

\begin{proposition}
Let $\beta \ge 3$ be a Pisot or Salem number, and let $\le$ be a consistent cylinder order on $\{0,\pm1\}^\infty$. 
Then $\pi_\beta(\bm^{\mathrm{abs}}_\le)$ is in $\mathbb{Q}(\beta)$ or  transcendental. 
\end{proposition}

\begin{proof}
Let $G(a_1a_2\cdots) = |a_1|\,|a_2|\cdots$. 
Then $G \circ \varrho_i = \tau_{0,1}$ for all $i \in \{0,1,2\}$, thus $p_{G(\bm^{\mathrm{abs}}_\le)}(n) \le 3n{-}2$ by Proposition~\ref{p:linear}, Theorems~\ref{t:main} and~\ref{t:sym}. 
Moreover, the map $G$ is 2-to-1 from the set of factors of~$\bm^{\mathrm{abs}}_\le$ to the set of factors of $G(\bm^{\mathrm{abs}}_\le)$, thus  $p_{\bm^{\mathrm{abs}}_\le}(n) \le 6n{-}4$.
By \cite[Theorem~1A]{Adamczewsk-Bugeaud07} and by adding 1 to each digit of $\bm^{\mathrm{abs}}_\le$, we obtain that $\pi_\beta(\bm^{\mathrm{abs}}_\le) + \frac{1}{\beta-1}$ is in $\mathbb{Q}(\beta)$ or transcendental, thus also $\pi_\beta(\bm^{\mathrm{abs}}_\le)$ is in $\mathbb{Q}(\beta)$ or transcendental.
\end{proof}

\section{Examples of orders on symmetric shift spaces} \label{sec:exampl-orders-symm}
\subsection{Lexicographic order}
For the lexicographic order on $\{0,\pm1\}^\infty$ (with $-1 < 0 < 1$), we have $[1\bar{1}] < [10]$, $[10\bar{1}] < [100]$, and we obtain that
\[
\bm^{\mathrm{abs}}_{\le_{\mathrm{lex}}} = \varrho_2(\bm_{\le_{\mathrm{alt}}}) = 10\bar{1}1\bar{1}010\bar{1}01\bar{1}10\bar{1}1\bar{1}01\bar{1}10\bar{1}010\bar{1}1\bar{1}010\cdots.
\]
The following proposition relates the Lagrange spectrum $\tilde{\mathcal{L}}_q$, defined in~\eqref{e:Lbeta}, and its smallest accumulation point~$\tilde{\bm}_q$ to $\mathcal{M}^{\mathrm{abs}}_{\le_{\mathrm{lex}}}$ and~$\bm^{\mathrm{abs}}_{\le_{\mathrm{lex}}}$; it slightly improves results of \cite{Dubickas06, Akiyama-Kaneko21}. 

\begin{proposition} \label{p:Lbeta}
We have 
\begin{gather}
\tilde{\mathcal{L}}_2 = \pi_2(\mathcal{M}^{\mathrm{abs}}_{\le_{\mathrm{lex}}}) \cap \big[0,\tfrac{1}{2}\big] \neq \pi_2(\mathcal{M}^{\mathrm{abs}}_{\le_{\mathrm{lex}}}), \quad \tilde{\mathcal{L}}_3 = \pi_3(\mathcal{M}^{\mathrm{abs}}_{\le_{\mathrm{lex}}}), \label{e:23} \\
\pi_q(\mathcal{M}^{\mathrm{abs}}_{\le_{\mathrm{lex}}}) = \tilde{\mathcal{L}}_q \cap \big[0,\tfrac{1}{q-1}\big] \neq \tilde{\mathcal{L}}_q \quad \mbox{for all integers}\ q \ge 4. \label{e:4}
\end{gather}
For all integers $q \ge 2$, we have $\tilde{\bm}_q = \pi_q(\bm^{\mathrm{abs}}_{\le_{\mathrm{lex}}}) = \pi_q(\varrho_2(\bm_{\le_{\mathrm{alt}}}))$ and
\[
\tilde{\mathcal{L}}_q \cap [0,\tilde{\bm}_q) = \pi_q(\mathcal{M}^{\mathrm{abs}}_{\le_{\mathrm{lex}}}) \cap [0,\tilde{\bm}_q) = \{0\} \cup \{\pi_q(\varrho_2(\tau_{0,2}^n(0^\infty))) \,:\, n \ge 0\}. 
\]
\end{proposition}

\begin{proof}
As mentioned at the beginning of Section~\ref{sec:symm-shift-spac}, for integer $q \ge 2$, $\|x q^n\|$ can be determined by a symmetric $q$-expansion of~$x$.
We can assume w.l.o.g.\  $|x| \le \frac{1}{2}$.
Let 
\[
\mathcal{A}_q := \big\{a_1a_2\cdots \in A_q^\infty \,:\, |\pi_q(a_ka_{k+1}\cdots)| \le \tfrac{1}{2} \ \mbox{for all}\ k \ge 1\big\},\ \mbox{with}\ A_q := \{0, \pm1, \dots, \pm\lfloor q/2\rfloor\}.
\]
For each $x \in [-\frac{1}{2},\frac{1}{2}]$, we obtain a sequence $\ba = a_1a_2\cdots \in \mathcal{A}_q$ satisfying $x = \pi_q(\ba)$ by taking $a_k = \lfloor q \tilde{T}_q^{k-1}(x) + \frac{1}{2}\rfloor$ where $\tilde{T}_q(y) := q y - \lfloor q y + \frac{1}{2}\rfloor$.
Then
\[
\|x q^n\| = |\pi_q(a_{n+1}a_{n+2}\cdots)| = \pi_q(\mathrm{abs}(a_{n+1}a_{n+2}\cdots)).
\]
Note that $\ba \le_{\mathrm{lex}} \bb$ implies $\pi_q(\ba) \le \pi_q(\bb)$ for all $\ba, \bb \in \mathcal{A}_q$. 
Since $\mathcal{L}^{\mathrm{abs}}_{\le_{\mathrm{lex}}} = \mathcal{M}^{\mathrm{abs}}_{\le_{\mathrm{lex}}}$ (and a similar relation holds for larger alphabets), we obtain that
\[
\tilde{\mathcal{L}}_q = \{\pi_q(s^{\mathrm{abs}}_{\le_{\mathrm{lex}}}(\ba)) \,:\, \ba \in \mathcal{A}_q\}.
\]
For $q \in \{2,3\}$, we have $A_q = \{0,\pm1\}$, thus $\tilde{\mathcal{L}}_q \subseteq \pi_q(\mathcal{M}^{\mathrm{abs}}_{\le_{\mathrm{lex}}})$.
For $q \ge 3$, we have $\{0,\pm1\}^\infty \subseteq \mathcal{A}_q$, thus $\pi_q(\mathcal{M}^{\mathrm{abs}}_{\le_{\mathrm{lex}}}) \subseteq \tilde{\mathcal{L}}_q$.
Since $\pi_2(\mathcal{M}^{\mathrm{abs}}_{\le_{\mathrm{lex}}}) \cap [0,\frac{1}{2}] \subseteq \tilde{\mathcal{L}}_2$ and $1 = \pi_2(1^\infty) \in \pi_2(\mathcal{M}^{\mathrm{abs}}_{\le_{\mathrm{lex}}}) \setminus \tilde{\mathcal{L}}_2$, this proves~\eqref{e:23}. 
For $q \ge 4$, we have $\pi_q(s^{\mathrm{abs}}_{\le_{\mathrm{lex}}}(\ba)) \ge \pi_q((2\overline{2})^\infty) = \frac{2}{q+1} > \frac{1}{q-1}$ for all $\ba \in A_q^\infty \setminus \{0,\pm1\}^\infty$, thus $\tilde{\mathcal{L}}_q \cap [0,\frac{1}{q-1}] \subseteq\pi_q(\mathcal{M}^{\mathrm{abs}}_{\le_{\mathrm{lex}}})$.
Together with $\frac{2}{q+1} \in \tilde{\mathcal{L}}_q \setminus \pi_q(\mathcal{M}^{\mathrm{abs}}_{\le_{\mathrm{lex}}})$, we obtain~\eqref{e:4}. 
Since $\pi_q$ is order-preserving on~$\mathcal{A}_q$, we obtain that $\tilde{\bm}_q = \pi_q(\bm^{\mathrm{abs}}_{\le_{\mathrm{lex}}})$ and that $\tilde{\mathcal{L}}_q$ and $\pi_q(\mathcal{M}^{\mathrm{abs}}_{\le_{\mathrm{lex}}})$ agree on $[0,\tilde{\bm}_q)$.
Since $\{\ba \in \mathcal{M}^{\mathrm{abs}}_{\le_{\mathrm{lex}}} \,:\, \ba < \bm^{\mathrm{abs}}_{\le_{\mathrm{lex}}}\}$ is equal to $\{0^\infty\} \cup \{\varrho_2(\tau_{0,2}^n(0^\infty)) \,:\, n \ge 0\}$ by Theorem~\ref{t:sym} and \eqref{e:discretealt}, this completes the proof of the proposition.
\end{proof}

\subsection{Alternating lexicographic order}
For the alternating lexicographic order on $\{0,\pm1\}^\infty$ (with $-1 < 0 < 1$), we have $[11] <_{\mathrm{alt}} [10]$ and $[10\bar{1}] <_{\mathrm{alt}} [100]$, 
\[
\bm^{\mathrm{abs}}_{\le_{\mathrm{alt}}} = \varrho_1(\bm_{\le_{\mathrm{alt}}}) = 10\bar{1}\bar{1}\bar{1}010\bar{1}01110\bar{1}\bar{1}\bar{1}01110\bar{1}010\bar{1}\bar{1}\bar{1}010\cdots.
\]

\subsection{Bimodal order}
Similarly to the unimodal order, we define the \emph{bimodal order} on $\{0,\pm1\}^\infty$ by $[wa] <_{\mathrm{bi}} [wb]$ if $a < b$ (with $-1 < 0 < 1$) and $|w|_1 + |w|_{-1}$ is even, or $a > b$ and $|w|_1 + |w|_{-1}$ is odd.
Then $\bm^{\mathrm{abs}}_{\le_{\mathrm{bi}}} = \tau_{0,1}(\bm_{\le_{\mathrm{alt}}}) = \bm_{\le_{\mathrm{uni}}}$. 
We get the same result for the order defined by $[wa] < [wb]$ if $a < b$ and $|w|_1$ is even, or $a > b$ and $|w|_1$ is odd.

We also define the \emph{flipped bimodal order} on $\{0,\pm1\}^\infty$ by $[wa] <_{\mathrm{biflip}} [wb]$ if $a < b$ and $|w|_0$ is even, or $a > b$ and $|w|_0$ is odd.
Then 
\[
\bm^{\mathrm{abs}}_{\le_{\mathrm{biflip}}} = \varrho_0(\bm_{\le_{\mathrm{alt}}}) = 101\bar{1}1010101\bar{1}101\bar{1}101\bar{1}1010101\bar{1}1010\cdots..
\] 

\subsection{Other orders}
For $\be \in \{\pm1\}^{\infty}$, we define a cylinder order~$\le_{\be}$ on $\{0,\pm1\}^\infty$ by 
\[
\ba \le_{\be} \bb \quad \mbox{if}\ \be \cdot \ba \le_{\mathrm{lex}} \be \cdot \bb,
\]
where $(e_1e_2\cdots) \cdot (a_1a_2\cdots) = (e_1a_1) (e_2a_2) \cdots$. 
We know from Proposition~\ref{p:transcendental} that $\pi_\beta(\bm^{\mathrm{abs}}_{\le_{\be}})$ is in $\mathbb{Q}(\beta)$ or transcendental for all Pisot or Salem numbers~$\beta$. 
However, here the value of $\pi_\beta(\be \cdot \bm^{\mathrm{abs}}_{\le_{\be}})$, which is the the smallest accumulation point of $\{\limsup_{n\to\infty} |\sum_{k=n+1}^\infty \frac{e_ka_k}{\beta^{k-n}}| \,:\, a_1a_2\cdots \in \{0,\pm1\}^\infty\}$
when $e_1=1$ and $\beta \ge 3$, is more relevant. 
If $\be$ is periodic with period~$k$, then $p_{\be\cdot\ba}(n) \le k\, p_{\ba}(n)$, hence $\pi_\beta(\be \cdot \bm^{\mathrm{abs}}_{\le_{\be}})$ is also in $\mathbb{Q}(\beta)$ or transcendental for all Pisot or Salem numbers~$\beta$. 
We do not know whether the same result holds when $\be$ is aperiodic. 

\bibliographystyle{amsalpha}
\bibliography{order}
\end{document}